\documentclass{amsart}
\usepackage{amsfonts} 
\usepackage{color}
\usepackage{graphicx}

\newtheorem{thm}{Theorem}[section]
\newtheorem{cor}[thm]{Corollary}
\newtheorem{lema}[thm]{Lemma}
\newtheorem{prop}[thm]{Proposition}
\theoremstyle{definition}
\newtheorem{defn}[thm]{Definition}
\theoremstyle{remark}
\newtheorem{rem}[thm]{Remark}

\numberwithin{equation}{section}
\newcommand{\R}{\mathbb R}
\newcommand{\N}{\mathbb N}

\newcommand{\ve}{\varepsilon}
\newcommand{\lam}{\lambda}

\newcommand{\cd}{\rightharpoonup}

\def\diver{\operatorname {\mathrm{div}}}

\begin{document}
\title[Homogenization]{Eigenvalue homogenization problem  \\ with indefinite weights}
\author[J Fern\'andez Bonder, J P Pinasco, A M Salort]{Juli\'an Fern\'andez Bonder, Juan P. Pinasco, Ariel M. Salort }
\address{Departamento de Matem\'atica \hfill\break \indent   IMAS - CONICET
 \hfill\break \indent FCEyN - Universidad de Buenos Aires
\hfill\break \indent Ciudad Universitaria, Pabell\'on I \hfill\break \indent   (1428)
Av. Cantilo s/n. \hfill\break \indent Buenos Aires, Argentina.}
\email[J. Fernandez Bonder]{jfbonder@dm.uba.ar}
\urladdr[J. Fernandez Bonder]{http://mate.dm.uba.ar/~jfbonder}
\email[J.P. Pinasco]{jpinasco@dm.uba.ar}
\urladdr[J.P. Pinasco]{http://mate.dm.uba.ar/~jpinasco}
\email[A.M. Salort]{asalort@dm.uba.ar}

\begin{abstract}
In this work we study the homogenization problem for nonlinear elliptic equations  involving $p-$Laplacian type operators with sign changing weights. We study the  asymptotic behavior of variational eigenvalues,  which consist on a double sequence of eigenvalues. We show that the $k-$th positive eigenvalue goes to infinity when the average of the weight is nonpositive, and converge to the $k-$th variational eigenvalue of the limit problem when the average is positive for any $k\ge 1$.
 \end{abstract}

\maketitle

\section{Introduction}

In this work we will consider the following  nonlinear eigenvalue problem with indefinite weight,
\begin{equation}\label{1.1}
\begin{cases}
-\diver(a_{\ve}(x,\nabla u)) = \lambda \rho_\ve(x) |u|^{p-2} u & \text{in }\Omega\subset \R^N\\
u=0 & \text{on } \partial\Omega,
\end{cases}
\end{equation}
where $\lambda$ is the eigenvalue parameter, $\rho_\ve$ is a bounded weight function with nontrivial positive and negative parts and the operator $\diver(a_\ve(x,\nabla u))$ is a quasilinear $(p-1)-$homogeneous in the second variable with some precise hypotheses that are stated below (see assumptions (H0)--(H8) in section 2). The most relevant example of such operator is given by
$$
a_\ve(x,\xi) = |A_\ve(x)\xi \cdot \xi|^{\frac{p-2}{2}} A_\ve(x)\xi,
$$
with $A_\ve(x)\in \R^{N\times N}$ bounded symmetric matrix and positive definite uniformly in $\ve>0$.


The domain $\Omega$ is assumed to be bounded but no regularity hypotheses are imposed on $\partial \Omega$.

For this eigenvalue problem \eqref{1.1}, it is known that there exist two sequences of eigenvalues $\{\lam_{\ve,k}^+\}_{k\ge 1}$, $\{\lam_{\ve,k}^-\}_{k\ge 1}$,  such that
$$
\lam_{\ve,k}^+ \to \infty,\quad \lam_{\ve,k}^- \to -\infty \qquad \mbox{as } k\to \infty.
$$
See  \cite{FBPS1} for a survey.

The asymptotic behavior as $\ve\downarrow 0$ of these nonlinear eigenvalues in the unweighted case (i.e. $\rho_\ve \equiv 1$) was studied in \cite{Con, Ch-dP, Chi1}. In particular, for the problem
$$
-\diver(a_{\ve}(x,\nabla u)) = \lambda |u|^{p-2} u \qquad \text{in }\Omega\\
$$
it is proved that the $k-$th variational eigenvalue converges to  the $k-$th variational eigenvalue of the limit problem
$$
-\diver(a(x,\nabla u)) = \lambda |u|^{p-2} u \qquad \text{in }\Omega\\
$$
where $a(x,\xi)$ is the so-called $G-$limit of the operators $a_\ve(x,\xi)$. See \cite{Ch-dP}. Of course, the convergence is understood up to a subsequence. See section 2 for the definition and some elemetary properties of the $G-$convergence.

The $G-$convergence of monotone operators has a long story and there are plenty of results in the literature due to the usefullness of this concept in the limit behavior of boundary value problems, specially in {\em homogenization theory}, see, for instance \cite{DF1992, Chi1, Fus} and references therein. 

The purpose in this paper is to extend the results of \cite{Ch-dP} to the undefinite weighted case. The main result of this work is the following.

\begin{thm}\label{teo.main}
Assume that $a_\ve(x,\xi)$ satisfies \em{(H0)--(H8)} defined is section 2. Moreover, assume that $a_\ve(x,\xi)$ $G-$converges to $a(x,\xi)$. Let $\rho_\ve\in L^\infty(\Omega)$ be such that $\rho_\ve \rightharpoonup \rho$ weakly* in $L^\infty(\Omega)$. Then
\begin{itemize}
\item If $\rho^+ = 0$, $\lambda_{k,\ve}^+\to\infty$ as $\ve\downarrow 0$.

\item If $\rho^+\neq 0$, $\lambda_{k,\ve}^+\to \lambda_k^+$ as $\ve\downarrow 0$, where $\{\lambda_k^+\}_{k\ge 1}$ are the positive eigenvalues associated to the operator $a(x,\xi)$ with weight $\rho$.
\end{itemize}
\end{thm}

Our approach follows closely the one in \cite{Ch-dP}. The main difference is the fact that we cannot work with a uniform normalization condition as in the unweighted case. The normalization condition varies with $\ve$ and that has to be taken care of. 

\begin{rem}
Obviously, an analogous statement holds for the negative eigenvalues with the obvious modifications.
\end{rem}

For second order linear elliptic operators, the eigenvalue convergence for the problem of periodic homogenization with sign changing weights was studied recently in \cite{Naz}. Our results here are closely related to theirs, although several differences arise. Of course, in our setting we are not able to use asymptotic expansions, nor orthogonality of eigenfunctions, So, our proofs are different, based mainly on the variational arguments developed in \cite{Ch-dP}. The main drawback of our approach is that we were unable to obtain one of their results, the convergence of the rescaled sequences of eigenvalues which diverges and the corresponding limit problem.

On the other hand, our hypotheses in $a_{\ve}$ go beyond periodic homogenization, and  we have relaxed the regularity hypotheses on $\Omega$, since in \cite{Naz} they work with domains  of class $C^{2,\alpha}$. Also, different boundary conditions can be handled in this way as in \cite{FBPS2}.

Now, a natural question is the study of a quantitative version of Theorem \ref{teo.main}. That is, give some precise rate of convergence or divergence of the eigenvalues.

Recently, in \cite{FBPS} we have obtained the rate of convergence of eigenvalues of problem \eqref{1.1} in the case where the operator $a_\ve(x,\xi)$ is independent of $\ve$ and the wheight function is positive and given in terms of a periodic function $\rho$, as $\rho_\ve(x) = \rho(\tfrac{x}{\ve})$.

See also the bibliography in \cite{FBPS} for references about the linear problem and \cite{FBPS2} for the analysis of different boundary conditions. Moreover, in \cite{Salort} the analysis for the Fu\v{c}ik eigenvalue problem for the $p-$laplacian is done and in \cite{Salort2} the fractional laplace operator is studied.

In order to perform such analysis, we need to make some further assumptions on the weights $\rho_\ve$ and on the operators $a_\ve(x,\xi)$.

First, we assume that the weights $\rho_\ve$ are given in terms of a periodic function $\rho$ in the form $\rho_\ve(x) = \rho(\tfrac{x}{\ve})$. In this case, it follows that
$$
\rho_\ve \rightharpoonup \bar \rho := \int_Y \rho(y)\, dy,
$$
where $Y=[0,1]^N$ and the funcion $\rho$ is assumed to be $Y-$periodic. Under these assumptions, we obtain a rate of divergence for the eigenvalues. More precisely, we prove

\begin{thm}\label{main}
Assume that $a_\ve(x,\xi)$ satisfies \em{(H0)--(H8)} defined is section 2. Assume, moreover that $\rho_\ve(x) = \rho(\tfrac{x}{\ve})$ where $\rho$ is a $Y-$periodic function, $Y=[0,1]^N$. Then\begin{enumerate}
\item if $\bar\rho = 0$  then $\lam_{\ve,1}^{\pm} = O(\ve^{-1})$  as $\ve\downarrow 0$,

\item  if $\bar\rho > 0$  then  $\ve \lam_{\ve,1}^-$ is bounded away from zero as $\ve\downarrow 0$,

\item if $\bar\rho < 0$  then   $\ve\lam_{\ve,1}^+$ is bounded away from zero as $\ve\downarrow 0$.
\end{enumerate}
Moreover, in the case of periodic homogenization, with $a_{\ve}(x, \xi) = a(\tfrac{x}{\ve}, \xi)$, we have that
\begin{enumerate}
\item if $\bar\rho = 0$  then $\lam_{\ve,k}^{\pm} = O(\ve^{-1}) $ as $\ve\downarrow 0$,

\item if $\bar\rho > 0$  then $\ve^p \lam_{\ve,k}^-$ is bounded away from infinity as $\ve\downarrow 0$,

\item if $\bar\rho < 0$  then $\ve^p \lam_{\ve,k}^+$ is bounded away from infinity as $\ve\downarrow 0$.
\end{enumerate}
\end{thm}

\begin{rem} In \cite{Naz}, where only linear eigenvalue problems and periodic homogenization was considered, the authors proved that $c_k^- \ve^{-2}\le \lam_{\ve,k}^- \le C_k^- \ve^{-2 }$ when $\bar \rho >0$. That result is obtained by using a factorization technique in order to construct the eigenfunctions asymptotic. We cannot use here these kind of arguments, due to the nonlinear character of the problem, and we get only the upper bound with a worse lower bound.
\end{rem}

Finaly, in the case where the operators $a_\ve(x,\xi)$ are independent of $\ve$ we can obtain a rate of convergence of the eigenvalyes. 
\begin{thm}\label{main2} Let $a_{\ve}(x,\nabla u)=a(x,\nabla u)$ be fixed, not depending on $\ve$, and $\rho_\ve$ as in Theorem \ref{main}. If $\bar \rho>0$, we have the following estimate,
$$
|\lam_{\ve,k}^{+}-\lam_k| \le C_k \ve,
$$
where $C_k$ is given explicitly, and depend only on $p$, $N$ and $\|\rho\|_{\infty}$.

An analogous result holds when $\bar\rho<0$.
\end{thm}

\subsection*{Organization of the paper}
After this introduction, the rest of the paper is organized as follows:

In section 2 we recall some preliminary results needed in the rest of the paper. In section 3 we prove the main result of the paper, namely Theorem \ref{teo.main}. In section 4 we prove our results on the divergence of eigenvalues, Theorem \ref{main}. Finally, in section 5, we prove the rate of convergence of the eigenvalues Theorem \ref{main2}.

\section{Preliminary results}
\subsection{$G-$convergence of monotone operators}

Let us consider the operator $\mathcal{A}\colon W_0^{1,p}(\Omega)\to
W^{-1,p'}(\Omega)$ given by
$$
\mathcal{A} u := -div(a(x,\nabla u)).
$$
 We assume that
$a\colon \Omega\times \R^N\to \R^N$ satisfies,  for every $\xi\in\R^N$ and  a.e. $x\in
\Omega$, the following conditions:

\begin{enumerate}
\item[(H0)] {\it measurability:} $a(\cdot,\cdot)$ is a Carath\'eodory function;
    that is, $a(x,\cdot)$ is continuous a.e. $x\in \Omega$,  and $a(\cdot,\xi)$ is
    measurable for every $\xi\in\R^N$.

\item[(H1)] {\it  monotonicity:} $0\le (a(x,\xi_1)-a(x,\xi_2))(\xi_1-\xi_2)$.

\item[(H2)] {\it  coercivity:} $\alpha |\xi|^p \le a(x,\xi)\cdot \xi$.

\item[(H3)] {\it  continuity:} $|a(x,\xi)| \le \beta|\xi|^{p-1}$.

\item[(H4)] {\it  $(p-1)-$homogeneity:} $a(x,t\xi)=t^{p-1} a(x,\xi)$ for every $t>0$.

\item[(H5)] {\it  oddness:} $a(x,-\xi) = -a(x,\xi)$.
\end{enumerate}

Let us introduce $\Psi(x,\xi_1, \xi_2)=a(x,\xi_1)\cdot \xi_1 + a(x,\xi_2)\cdot \xi_2$ for all
$\xi_1, \xi_2 \in \R^N$, and all $x\in \Omega$; and let $\delta=min\{p/2, (p-1)\}$.

\begin{enumerate}
\item[(H6)] {\it  equi-continuity:} 
$$
|a(x,\xi_1) -a(x,\xi_2)| \le c \Psi(x,\xi_1, \xi_2)^{(p-1-\delta)/p}((a(x,\xi_1) -a(x,\xi_2))\cdot (\xi_1-\xi_2))^{\delta/p}
$$

\item[(H7)] {\it  cyclical monotonicity:} $\sum_{i=1}^k  a(x,\xi_i)\cdot (\xi_{i+1}-\xi_i) \le 0$, for all $k\ge 1$, and $\xi_1,\dots, \xi_{k+1}$, with $\xi_1=\xi_{k+1}$.

\item[(H8)] {\it  strict monotonicity:} let $\gamma = \max(2,p)$, then
$$
\alpha |\xi_1-\xi_2|^{\gamma}\Psi(x,\xi_1,\xi_2)^{1-(\gamma/p)}\le (a(x,\xi_1)-a(x,\xi_2))\cdot (\xi_1-\xi_2).
$$
\end{enumerate}

Under these conditions  $\mathcal{A}$ is a monotone operator. Moreover, we have the
following results:

\begin{prop}[\cite{Con}, Lemma 3.3]\label{potential.f}
Given $a(x,\xi)$ satisfying  {\em (H0)--(H8)} there exists a unique Carath\'eodory function
$\Phi$ which is even, $p-$homogeneous strictly convex and differentiable in the
variable $\xi$ satisfying
\begin{equation}\label{cont.coer.phi}
\alpha |\xi|^p \le \Phi(x,\xi)\le \beta |\xi|^p
\end{equation}
for all $\xi\in \R^N$ a.e. $x\in\Omega$ such that
$$
\nabla_{\xi} \Phi(x,\xi)= p\, a(x,\xi)
$$
and normalized such that $\Phi(x,0)=0$.
\end{prop}

Let us recall the definition of $G-$ and Mosco-convergence:

\begin{defn}\label{defigconv}
We say that the family of operators $\mathcal{A}_\ve u := -\diver(a_\ve(x,\nabla u))$
$G$-converges to $\mathcal{A}u:=-\diver(a(x,\nabla u))$ if for every $f\in
W^{-1,p'}(\Omega)$  and for every $f_\ve$ strongly convergent to $f$ in
$W^{-1,p'}(\Omega)$, the solutions $u^\ve$ of the problem
\begin{equation*}
\begin{cases}
-\diver(a_\ve(u^\ve,\nabla u^\ve))=f_\ve &\textrm{ in } \Omega \\
u^\ve=0 & \textrm{ on } \partial \Omega
\end{cases}
\end{equation*}
satisfy the following conditions
\begin{align*}
 u^\ve \cd u  &\qquad \mbox{ weakly in }  W^{1,p}_0(\Omega), \\
 a_\ve(x, \nabla u^\ve) \cd a(x,\nabla u)  &\qquad \mbox{ weakly in }  (L^{p}(\Omega))^n,
\end{align*}
where $u$ is the solution to the equation
\begin{equation*}
\begin{cases}
-\diver(a(x,\nabla u))=f & \textrm{ in } \Omega  \\
u=0 &\textrm{ on } \partial \Omega.
\end{cases}
\end{equation*}
\end{defn}

\begin{defn}
Let $X$ be a reflexive Banach space and  $F_j:X\to [0,+\infty]$ be a sequence of
functionals on $X$. Then, $F_j$ \textit{Mosco-converge} to $F$  if and only if the
following conditions hold:
 \begin{itemize}
\item Lower bound inequality: For every sequence $\{u_j\}_{j\ge 1}$ such that $u_j
    \cd u$ weakly in $X$ as $j\to\infty$, $$ F(u) \le \liminf_{j\to\infty}
    F_j(u_j).
    $$

    \item Upper bound inequality: For every $u\in X$, there exists  a sequence
        $\{u_j\}_{j\ge 1}$ such that $u_j \to u$ strongly in $X$ as $j\to\infty$
        such that
        $$ F(u) \ge
    \limsup_{j\to\infty} F_j(u_j).
    $$
\end{itemize}
\end{defn}

In the general case, one has the following results proved in \cite{Chi1, Con}
\begin{thm}[\cite{Chi1}, Theorem 4.1]\label{Gconv.ve}
Assume that $a_\ve(x,\xi)$ satisfies {\em (H1)--(H3)}. Then, up to a subsequence, $\mathcal{A}_\ve$ $G-$converges to a maximal monotone operator $\mathcal{A}$ whose coefficient $a(x,\xi)$ also satisfies {\em (H1)--(H3)}.
\end{thm}

Moreover,
\begin{thm}[\cite{Con}, Theorem 2.3]
If $\mathcal{A}_\ve u := -\diver(a_\ve(x,\nabla u))$ $G-$converges to $\mathcal{A}u:=-\diver(a(x,\nabla u))$ and $a_\ve(x,\xi)$ satisfies {\em (H0)--(H8)}, then $a(x,\xi)$ also satisfies {\em (H0)--(H8)}.
\end{thm}

\begin{lema}[\cite{Con}, Lemma 4.2]\label{lema.mosco}
Given $a_{\ve}(x,\xi)$, $a(x,\xi)$ satisfying {\em (H0)--(H8)}, and $\Phi_\ve(x,\xi)$, $\Phi(x,\xi)$ given by Proposition \ref{potential.f}. Let  $F_{\ve}$ and $F$ be defined as
$$
F_{\ve}(u)=\begin{cases}
\int_{\Omega} \Phi_{\ve}(x,\nabla u)\, dx & \quad u\in W_0^{1,p}(\Omega), \\
+\infty & \text{otherwise},
\end{cases}
$$
$$
F(u)=\begin{cases}
\int_{\Omega} \Phi(x,\nabla u)\, dx & \quad u\in W_0^{1,p}(\Omega), \\
+\infty & \text{otherwise}.
\end{cases}
$$
If $\mathcal{A}_\ve$ $G-$converges to $\mathcal{A}$, then $F_{\ve}$ Mosco-converges to $F$.
\end{lema}

\subsection{Oscillatory integrals}

The proof of the main Theorem makes use of some results on convergence of oscillatory integrals.
In the case of periodic oscillations, the result needed here was proved in  \cite{FBPS}. In fact, in  \cite{FBPS} the following result is proved.

\begin{thm}[\cite{FBPS}, Lemma 3.3] 
Let $\Omega\subset \R^N$ be a bounded domain and denote by $Q$ the unit cube in
$\R^N$. Let $g\in L^\infty(\R^N)$ be a $Q$-periodic function such that $\bar g=0$.
Then the inequality
$$
\left| \int_{\Omega} g(\tfrac{x}{\ve})u \, dx\right| \leq \|g\|_{L^\infty(\R^N)} c_1\ve\|\nabla u\|_{L^1(\Omega)}
$$
holds for every $u\in W^{1,1}_0(\Omega)$, where $c_1$  is the optimal constant in Poincar\'e's inequality in $L^1(Q)$.
\end{thm}

For $v\in W^{1,p}_0(\Omega)$ if one applies the previous result to $u=|v|^p$ the next corollary is obtained.
\begin{cor}[\cite{FBPS}, Theorem 3.4]\label{lema.clave}
Under the same assumptions of the previous result, for any $v\in W^{1,p}_0(\Omega)$ it holds,
$$
\left| \int_{\Omega} g(\tfrac{x}{\ve})|v|^p \, dx\right| \leq \|g\|_{L^\infty(\R^N)} p c_1\ve\|v\|_{L^p(\Omega)}^{p-1} \|\nabla v\|_{L^{p}(\Omega)}.
$$
\end{cor}

In the general case, when no periodicity is assumed, one cannot have a rate of convergence. Nevertheless the following result holds.
\begin{thm}\label{no.oscilatorio}
Let $\Omega\subset \R^N$ be a bounded domain and let $\rho_\ve, \rho\in L^\infty(\Omega)$ be such that $\rho_\ve\rightharpoonup \rho$ *-weakly in $L^\infty(\Omega)$. Let $K\subset L^1(\Omega)$ be a compact set. Then, 
$$
\lim_{\ve\to 0} \sup_{v\in K} \int_\Omega (\rho_\ve - \rho) v\, dx = 0.
$$
\end{thm}

\begin{proof}
Given $r>0$, there exists $\{v_i\}_{i=1}^J\subset K$ such that $K\subset \cup_{i=1}^J B_r(v_i)$. By hypotheses, it holds that
$$
\lim_{\ve\to 0} \max_{1\le i\le J} \int_\Omega (\rho_\ve - \rho) v_i\, dx = 0.
$$

Let now $v_\ve\in K$ be such that
$$
\sup_{v\in K}\int_\Omega (\rho_\ve - \rho) v\, dx \le \int_\Omega (\rho_\ve - \rho) v_\ve \, dx + \ve. 
$$
Hence, there exists $i_\ve\in \{1,\dots,J\}$ such that $v_\ve\in B_r(v_{i_\ve})$. Now
\begin{align*}
\int_\Omega (\rho_\ve - \rho) v_\ve\, dx &= \int_\Omega (\rho_\ve - \rho) v_{i_\ve}\, dx + \int_\Omega (\rho_\ve - \rho) (v_\ve - v_{i_\ve})\, dx\\
&\le \max_{1\le i\le J} \int_\Omega (\rho_\ve - \rho) v_i\, dx + M r,
\end{align*}
where $M$ is a bound on $\|\rho_\ve\|_\infty + \|\rho\|_\infty$. Therefore
$$
\limsup_{\ve\to 0}\sup_{v\in K}\int_\Omega (\rho_\ve - \rho) v\, dx  \le Mr.
$$
Since $r>0$ is arbitrary, the result follows.
\end{proof}

In our application of Theorem \ref{no.oscilatorio}, the compact set will be a bounded set in $W^{1,1}_0(\Omega)$. So we have
\begin{cor}
Under the same hypotheses of the previous theorem, we have
$$
\lim_{\ve\to 0} \sup_{v\in W^{1,1}_0(\Omega),\atop \|\nabla v\|_1\le 1} \int_\Omega (\rho_\ve - \rho) v\, dx = 0.
$$
\end{cor}

Finally, when one consider bounded sets in $W^{1,p}_0(\Omega)$ the following analog of Corollary \ref{lema.clave} holds
\begin{cor}\label{lema.clave.2}
Under the same assumptions of Theorem \ref{no.oscilatorio}, 
$$
\left|\int_\Omega (\rho_\ve-\rho) |v|^p\, dx\right|\le o(1) \|v\|_p^{p-1} \|\nabla v\|_p.
$$
\end{cor}

\subsection{Eigenvalues of quasilinear operators}

We refer the interested reader to the survey  \cite{FBPS1} for details, only the facts
that will be used below are stated here.

In this subsection we state some results for the eigenvalue problem \eqref{1.1} for fixed $\ve>0$. Than is, we analyze the problem
\begin{equation}\label{eigen.problem}
\begin{cases}
-\diver(a(x,\nabla u)) = \lambda \rho(x) |u|^{p-2}u & \text{in }\Omega\\
u=0 & \text{on }\partial\Omega,
\end{cases}
\end{equation}
where $\Omega$ is a bounded open set in $\R^N$.

We assume that $a(x,\xi)$ satisfies (H0)--(H8) of the previous subsection, and as a consequence, there exists a potential funciont $\Phi(x,\xi)$ given by Proposition \ref{potential.f}. 

By using the Ljusternik-Schnirelmann theory, if $\rho^+\neq 0$, one can construct a sequence of (variational) eigenvalues of \eqref{eigen.problem} as 
\begin{equation}\label{variac}
\lam_{k}^+ = \inf_{C\in \mathcal{C}_k} \sup_{v\in C}
\frac{ \int_\Omega \Phi(x,\nabla v) \, dx}{ \int_\Omega \rho(x) |v|^p\, dx }
\end{equation}
where
\begin{align*}
\mathcal{C}_k & = \{ C \subset M^+ \colon C  \text{ is compact and symmetric, } \gamma(C) \ge k \},\\
M^+ & = \big\{u \in W^{1,p}_0(\Omega)\colon \int_{\Omega} \rho(x) |u|^p\, dx > 0 \big\},
\end{align*}
and $\gamma \colon \Sigma \to \N \cup \{\infty\}$ is the  Krasnoselskii genus, see \cite{Rab2},
$$
\gamma(A) = \min\{ k \in \N \colon \text{ there exists } f \in C(A, \R^k\setminus \{0\}), \ f(x) = -f(-x) \}.
$$
Of course, when $\rho^-\neq 0$ one can construct a seqence of negative eigenvalues in a complete analogous way changing $M^+$ by $M^-$ given by
$$
M^- = \big\{u \in W^{1,p}_0(\Omega)\colon \int_{\Omega} \rho(x) |u|^p\, dx < 0 \big\}.
$$

It is customary to reformulate \eqref{variac} as
$$  
\frac{1}{\lambda_k^+} = \sup_{C\in \mathcal{C}_k} \inf_{v\in C} \frac{ \int_\Omega \rho(x) |v|^p\, dx }{ \int_\Omega \Phi(x,\nabla v)\, dx},
$$
and due to the homogeneity condition (H4), we will use also the following equivalent characterization for the eigenvalues:
$$
\frac{1}{\lambda_k^+} = \sup_{C\in \mathcal{D}_k} \inf_{u\in C} \int_{\Omega} \rho(x) |u|^p dx,
$$
where
$$
\mathcal{D}_k=\{C \subset S\colon \text{ compact and symmetric with } \gamma(C)\ge k, \mbox{ and } 0\not\in C\},
$$
$$
S = M^+ \cap \left\{ u \in W_0^{1,p}(\Omega) \colon \int_{\Omega}\Phi(x,\nabla u)\, dx= 1\right\}.
$$

The following useful Sturm-type theorem will be needed later:

\begin{thm} \label{sturmian}
Let $\Omega_1\subset \Omega_2\subset \R^N$, and let  $\{\lam_{k, i}^{+}\}_{k\ge 1}$
be the eigenvalues given by \eqref{variac} in $\Omega_i$, $i=1,2$, respectively. Then,
$$
\lam_{k,2}^{+} \le \lam_{k,1}^{+}
$$
for any $k\ge 1$.

Moreover, let $\rho_1(x) \le \rho_2(x)$ a.e. $x\in \Omega$ and $\Phi_1(x,\xi)\ge \Phi_2(x,\xi)$ a.e. $x\in \Omega$, for every $\xi\in \R^N$. Then, if $\{\lam_{k,i}^{+}\}_{k\ge 1}$ are the eigenvalues given by \eqref{variac} with weight $\rho_i$ and potential $\Phi_i$, $i=1,2$ respectively, then,
$$
\lam_{k,2}^{+} \le \lam_{k,1}^{+}.
$$
\end{thm}

The proof follows easily by comparing the Rayleigh quotient and using the inclusion of
Sobolev spaces  $W_0^{1,p}(\Omega_1) \subset W_0^{1,p}(\Omega_2)$.

\section{Proof of the main result}

In this section we prove the main result of the paper, namely Theorem \ref{teo.main}.

Assume first that $\rho^+=0$ and $\rho_\ve^+\neq 0$ for every $\ve>0$. Let $u\in W^{1,p}_0(\Omega)$, then we have
$$
\int_{\Omega} \Phi_\ve(x,\nabla u)\, dx \ge \alpha \|\nabla u\|_p^p.
$$
On the other hand, by Corollary \ref{lema.clave.2}, we have
$$
\int_{\Omega} \rho_\ve |u|^p\, dx = \int_{\Omega} \rho |u|^p + o(1) \|u\|_p^{p-1} \|\nabla u\|_p \le o(1) \|u\|_p^{p-1} \|\nabla u\|_p.
$$
Hence, we get the bound
$$
\frac{\int_{\Omega} \Phi_\ve(x,\nabla u)\, dx}{\int_{\Omega} \rho_\ve |u|^p\, dx}\ge \frac{\alpha}{o(1)} \left(\frac{\|\nabla u\|_p^p}{\|u\|_p^p}\right)^{\frac{p-1}{p}}. 
$$
Now, taking the infimum in the former inequality, we obtain
$$
\lambda_{1,\ve}^+ \ge \frac{\alpha}{o(1)} \mu_1^{\frac{p-1}{p}},
$$
where $\mu_1$ is the first eigenvalue of the $p-$laplacian with Dirichlet boundary conditions.

From this, the first part of the theorem follows.

\medskip

Now, assume that $\rho^+\neq 0$. Then, $\rho_\ve^+\neq 0$ for every $\ve>0$ small enough.

%
%
%
%

Let us fix $\delta > 0$, and let $C\in \mathcal C_k$ be such that
$$
\sup_{v\in C}\int_\Omega \Phi(x,\nabla v) \,dx \le \lam_k^+ + \delta,
$$
with 
$$
\int_\Omega \rho(x) |v|^p\, dx =1 \quad \text{for any } v\in C.
$$
This last condition can be imposed without loss of generality by the homogeneity of $\Phi$.

Since $C$ is a compact set,  we can choose $r>0$ and $\{u_i\}_{i=1}^J \subset C$, $J=J(r)$,  such
that
\begin{align*}
C \subset \bigcup_{i=i}^J B(u_i, r), &\\
\left|\int_{\Omega}  |u - u_i|^p\, dx\right| <\frac{\delta}{\|\rho \|_{\infty}} & \qquad \mbox{ if }  u\in B(u_i, r),
\end{align*}

Since $\rho_\ve \stackrel{*}{\cd} \rho$, there exists some $\ve_0$ such that
$C\subset M^+_{\ve}$ for $0<\ve<\ve_0$.

By Lemma \ref{lema.mosco} we have that the functionals $\{F_\ve\}_{\ve>0}$ Mosco-converge to $F$ and this implies that, for any $1\le i\le J$, there exists a sequence $u_{\ve_j, i}\cd u_i$ such that
$$
\int_\Omega \Phi (x,\nabla u_i) \, dx = \lim_{j\to \infty} \int_\Omega \Phi_{\ve_j}(x,\nabla u_{\ve_j, i})\, dx.
$$
Moreover, we can assume that $u_{\ve_j, i}\to u_i$ in $L^p(\Omega)$, and thus  
$$
1-\delta\le \int_{\Omega}\rho_{\ve_j}(x) |u_{\ve_j,i}|^p \,dx \le 1+\delta.
$$

Following \cite{Ch-dP}, let us take $C_{\ve_j}$ the convex closure of $\{\pm
u_{\ve_j,i}\}_{i=i}^J$, which is a compact convex set (since it has dimension lower
than or equal to $J$). Observe that, since the functions $u_{\ve_j, i}$ are weakly convergent and hence bounded in $W^{1,p}_0(\Omega)$, the sets $C_{\ve_j}$ are bounded in $W^{1,p}_0(\Omega)$ uniformly in $\ve_j$.

We define the projection $P_{\ve_j} : C\to C_{\ve_j}$, and let us observe that, for any
$v\in C$, since $v\in B(u_i, r)$ for some $i$, we have
\begin{align*}
\|P_{\ve_j}(v)-v\|_p \le & \|P_{\ve_j}(u_i)  - v \|_p
 \\ \le &  \|P_{\ve_j}(u_i)  - v  +    u_{\ve_j, i} - u_{\ve_j, i}  \|_p
 \\
\le & \|P_{\ve_j}(u_i)  -     u_{\ve_j, i}\|_p + \| v   - u_{\ve_j, i}  \|_p
\\
\le & C \delta + r.
\end{align*}
On the other hand, since $\int_\Omega |v|^p \rho(x)\, dx=1$, we have that
$$
\int_\Omega |v|^p\, dx \ge \frac{1}{\|\rho\|_\infty} \int_\Omega |v|^p\rho(x)\, dx = \frac{1}{\|\rho\|_\infty}.
$$
Therefore, we conclude that
$$
\|P_{\ve_j}(v)\|_p \ge  \frac{1}{\|\rho\|_\infty^{\frac{1}{p}}} - (C\delta + r) \ge \theta>0
$$
and $G_{\ve_j} := P_{\ve_j}(C) \subset C_{\ve_j} \setminus B_\theta(0)$, and it has genus greater than or equal to $k$ for $\ve$ small enough. Again, the sets $G_{\ve_j}$ are uniformly bounded in $W^{1,p}_0(\Omega)$.

Now,
\begin{align*}
\lam_{\ve_j, k}^+ \le &  \sup_{v\in G_{\ve_j}} \frac{  \int_\Omega \Phi_{\ve_j}(x,\nabla v)\, dx}
{\int_{\Omega} \rho_{\ve_j}(x)  |v|^p\, dx}
 \\
  \le &
 \sup_{v\in G_{\ve_j}}\left(  \int_\Omega \Phi_{\ve_j}(x,\nabla v) \,dx\right)  \sup_{v\in G_{\ve_j}} \left(\frac{1}{\int_{\Omega}[\rho_{\ve_j}(x) - \rho(x)] |v|^p  + \rho(x)  |v|^p \, dx } \right)
 \\
  \le &  (1+ O(r)+O(\delta) + O(\ve_j)) \max_{1\le i\le J}   \int_\Omega \Phi_{\ve_j}(x,\nabla u_{\ve_j,i})  \, dx  
  \\
  \le & (1+ o(1)) \max_{1\le i\le J}   \int_\Omega \Phi(x,\nabla u_{i})  \,dx    \\
  \le & (1+ o(1))(\lam_k + \delta).
\end{align*}

Therefore, we have obtained the inequality
\begin{equation}\label{lowerboundlam}
 \lam_{\ve_j,k}^+ \le (1+o(1)) (\lam_k^+ + O(\delta)).
\end{equation}

Observe that, in particular, the sequence $\{\lambda_{\ve_j, k}^+\}_{j\in\N}$ is bounded for each $k\in \N$ and that
$$
\limsup_{j\to\infty} \lambda_{\ve_j,k}^+ \le \lambda_k^+.
$$

In order to prove the reverse inequality, that is $\lam_k^+ \le \liminf_{j\to\infty} \lam_{\ve_j,k}^+$,  let us start now with a family of compact sets $C_{\ve_j}\subset \{u\in W^{1,p}_0(\Omega) : \|\nabla
u\|_p=1\}$, and we choose $u_{\ve_j} \in C_{\ve_j}$ such that
$$
\frac{\int_\Omega \Phi_{\ve_j}(x,\nabla u_{\ve_j})  \,dx}{\int_{\Omega} \rho_{\ve_j}(x) |u_{\ve_j}|^p} = \sup_{v\in C_{\ve_j}} \frac{\int_\Omega \Phi_{\ve_j}(x,\nabla v)  \,dx}{\int_{\Omega} \rho_{\ve_j}(x) |v|^p} \le \lam_{\ve_j,k}^+ + \ve_j.
$$

We can extract a  sequence that we still denote by $\{u_{\ve_j}\}_{j\in \N}$, such that $u_{\ve_j} \rightharpoonup u_0$ in $W^{1,p}_0(\Omega)$ and so $|u_{\ve_j}|^p\to |u_0|^p$ in $L^1(\Omega)$.

Now, due to the Mosco-convergence of the functionals, $\Phi_{\ve}\to \Phi$, and the
weak convergence $\rho_\ve\stackrel{*}{\rightharpoonup} \rho$, we have
$$
\frac{\int_\Omega \Phi (x,\nabla u_0)  \,dx}{\int_{\Omega} \rho(x) |u_0|^p  \,dx} \le \liminf_{j \to \infty} \frac{ \int_\Omega \Phi_{\ve_j}(x,\nabla u_{\ve_j})  \,dx}{\int_{\Omega} \rho_{\ve_j}(x) |u_{\ve_j}|^p \,dx}  \le \liminf_{j \to \infty} \lam_{\ve_j,k}^+ + \ve_j < \infty.
$$

On the other hand, we can choose $v_{\ve_j}\in C_{\ve_j}$ such that
$$
\sup_{v\in C_{\ve_j}} \frac{ \int_\Omega \Phi(x,\nabla v)  \,dx}{ \int_\Omega \rho(x) |v|^p } \le
\frac{\int_\Omega \Phi(x,\nabla v_{\ve_j})  \,dx}{ \int_\Omega \rho(x) |v_{\ve_j}|^p } + \ve_j.
$$
We can extract some weakly convergent sequence denoted again by $\{v_{\ve_j}\}_{j\in \N}$,
such that $v_{\ve_j} \cd v_0$ in $W^{1,p}_0(\Omega)$.

Let us now show that
\begin{equation}\label{v0u0}
\frac{ \int_\Omega \Phi(x,\nabla v_0)  \,dx}{ \int_\Omega \bar\rho |v_0|^p\, dx}\le \frac{ \int_\Omega \Phi(x,\nabla u_0)  \,dx}{ \int_\Omega \bar\rho |u_0|^p \, dx}
\end{equation}

In fact, assume by contradiction that \eqref{v0u0} does not hold. Then there exists $\eta>0$ such that
$$
\frac{ \int_\Omega \Phi(x,\nabla v_0)  \,dx}{ \int_\Omega \bar\rho |v_0|^p \, dx} - \frac{ \int_\Omega \Phi(x,\nabla u_0)  \,dx}{ \int_\Omega \bar\rho |u_0|^p \, dx}>\eta.
$$
Hence
\begin{align*}
\liminf_{j\to\infty} \lambda_{\ve_j,k}^+ \ge \liminf_{j\to\infty} \frac{ \int_\Omega \Phi_{\ve_j}(x,\nabla v_{\ve_j})  \,dx}{\int_\Omega  \rho_{\ve_j}(x) |v_{\ve_j}|^p } \ge \frac{ \int_\Omega \Phi(x,\nabla v_0)  \,dx}{ \int_\Omega \bar\rho |v_0|^p }
>
\lam +\frac{\eta}{2},
\end{align*}
which is not possible, since
 $$
\lam= \lim_{\ve\to 0} \sup_{v\in C_{\ve}} \frac{ \int_\Omega \Phi_{\ve}(x,\nabla v)  \,dx}{
 \int_\Omega  \rho(\tfrac{x}{\ve}) |v|^p } \ge
 \frac{ \int_\Omega \Phi_{\ve}(x,\nabla v_{\ve})  \,dx}{ \int_\Omega
\rho(\tfrac{x}{\ve}) |v_{\ve}|^p } > \lam +\frac{\eta}{2},
$$
a contradiction. Thus,
$$\sup_{v\in C_{\ve}}
\frac{ \int_\Omega \Phi(x,\nabla v)  \,dx}{ \int_\Omega \bar\rho |v|^p } \le \lam.$$

Finally,
 \begin{align*}
 \lam_k = & \inf_{C\in \mathcal{C}_k} \sup_{v\in C} \frac{ \int_\Omega
\Phi(x,\nabla v)  \,dx}{ \int_\Omega \bar\rho |v|^p } \\
\le & \inf_{\ve_j}\sup_{v\in
C_{\ve_j}} \frac{ \int_\Omega \Phi(x,\nabla v)  \,dx}{ \int_\Omega \bar\rho |v|^p } \\
\le & \lam \\
\le & (1+o(1)) \lam_{\ve,k}
\end{align*}
since $\lam =  \lim_{j \to \infty} \lam_{\ve_j,k}$,
and by using  inequality
\eqref{lowerboundlam}, the proof is finished.

\section{Proofs of the divergence results}

\begin{proof}[Proof of Theorem \ref{main}]
We divide the proof in several parts. 

\begin{itemize}
\item[1.-] Case $\bar \rho =0$. First eigenvalue, lower bound.
\end{itemize}
It is enough to consider only the first positive eigenvalue $\lam_{\ve, 1}^+$, the
result for $\lam_{\ve, 1}^-$ follows by considering the weight $-\rho$.

We can bound $\lam_{\ve, 1}^+$ by below as follows:
\begin{equation} \label{z1}
\begin{split}
\frac{1}{\lam_{\ve, 1}^+} &= \sup_{v \in W_0^{1,p}} \left(\frac{ \int_\Omega \rho(\tfrac{x}{\ve}) |v|^p \, dx}{
 \int_\Omega \Phi(x,\nabla v)\, dx } \right)
 \\ \\
&\le \frac{\ve c_1 p}{\beta}  \|\rho\|_{L^\infty(\R^N)} \sup_{v \in W_0^{1,p}}
\left(\frac{ \|u\|_{L^p(\Omega)}^{p-1} \|\nabla u\|_{L^p(\Omega)} }{
 \int_\Omega |\nabla v|^p dx } \right)
 \\ \\
 &\le \ve C(p, \rho, c_1, \beta)  \sup_{v \in W_0^{1,p}}  \left( \frac{
  \int_\Omega |v|^p\, dx }{ \int_\Omega |\nabla v|^p\, dx } \right)^{p-1}
    \\
 \\ &\le \ve C(p, \rho, c_1, \beta, |\Omega|),
\end{split}
\end{equation}
where we have used Theorem \ref{lema.clave} in the first inequality.

Here the constant $C(p, \rho, c_1, \beta, |\Omega|)$ is obtained from Theorem
\ref{lema.clave} and then using the isoperimetric inequality, i.e., the first
eigenvalue in $\Omega$ is greater than the first eigenvalue of a ball $B_{|\Omega|}$
with the same measure than $\Omega$,
 $$
\sup_{v \in W_0^{1,p}}  \frac{\int_\Omega |v|^p \, dx}{\int_\Omega |\nabla v|^p\, dx} \le
\sup_{v \in W_0^{1,p}}  \frac{\int_{B_{|\Omega|}} |v|^p\, dx}{\int_{B_{|\Omega|}} |\nabla v|^p\, dx}
 $$

Therefore, the first positive eigenvalue goes to $+\infty$ at least as $\ve^{-1}$.

\begin{itemize}
\item[2.-] Case $\bar \rho =0$. First eigenvalue, upper bound.
\end{itemize}

The upper bound follows by taking as test function $v = (u^p +\ve u^p
|\rho(\tfrac{x}{\ve})|^{p-2}\rho(\tfrac{x}{\ve}))^{1/p}$, with a positive function
$u\in C_c^{\infty}(\Omega)$, and we get
\begin{equation} \label{z2}
\begin{split}
\frac{1}{ \lam_{\ve, 1}^+ } &= \sup_{v \in W_0^{1,p}} \left( \frac{ \int_\Omega \rho(\tfrac{x}{\ve}) |v|^p\, dx
  }{\int_\Omega \Phi(x,\nabla v)\, dx} \right) \\
&\ge\frac{ \int_\Omega u^p \rho(\tfrac{x}{\ve}) \left(1 +\ve
  |\rho(\tfrac{x}{\ve})|^{p-2}\rho(\tfrac{x}{\ve})\right)\, dx }{
   \alpha  \int_\Omega |\nabla (u^p +\ve u^p
|\rho|^{p-2}\rho)^{1/p}|^p\, dx }\\
&\ge   C(u, \rho, p, \alpha)  \int_\Omega \ve
u^p |\rho(\tfrac{x}{\ve})|^{p}\,   dx
\end{split}
\end{equation}
where the constant $C$  is strictly positive for $\ve$ small enough.

Hence, we have proved that $\lam_{\ve,1}^{+} = O(\ve^{-1})$.

\begin{itemize}
\item[3.-] Case $\bar \rho =0$. Higher eigenvalues, lower bound.
\end{itemize}

This is immediate from Step 1, since $\lambda_{\ve, k}^+\ge \lambda_{\ve, 1}^+\ge C\ve^{-1}$.

\begin{itemize}
\item[4.-] Case $\bar \rho =0$. Higher eigenvalues, upper bound.
\end{itemize}

Let $k\in\N$ be fixed and fix $\ve_0>0$ small such that there exists $\{Q_i\}_{i=1}^k$ where $Q_i\subset \Omega$ is a cube of side lenght $\ve_0$ and $Q_i\cap Q_j=\emptyset$ (we consider open cubes).

By the scaling properties of the eigenvalues, we have that $\mu_{1,\ve}(Q_i) = \ve_0^{-p}\mu_{1,\frac{\ve}{\ve_0}}(Q_0)$, where $\mu_{1,\ve}(U)$ is the first eigenvalue of the $p-$laplacian in $U\subset \R^N$ with Dirichlet boundary conditions and weight $\rho_\ve$.

Let $u_i$ be the first eigenfunction corresponding to $\mu_{1,\ve}(Q_i)$ and extended by 0 to $\Omega$, and let us define the set
$$
C_k = span\{ u_i : 1\le i \le k\} \cap B,
$$
where $B$ is the unit ball in $W_0^{1,p}(\Omega)$. Clearly, $C_k$ is a $k-$dimensional set, since the functions have disjoint support, and hence $\gamma(C_k)=k$.

By taking an arbitrary element $v=\sum b_i u_i \in C_k$, we  get
\begin{align*}
\frac{\int_{\Omega} a_{\ve}(x,\nabla v)\cdot \nabla v\, dx}{\int_{\Omega} \rho(\tfrac{x}{\ve})|v|^p\, dx}&=
\frac{ \sum_{i=1}^k  \int_{Q_i}  a_{\ve}(x,b_i\nabla u_i)\cdot b_i \nabla u_i \, dx
}{ \sum_{i=1}^k  \int_{Q_i} \rho(\tfrac{x}{\ve})| b_i u_i|^p\, dx}\\
&\le \beta \frac{ \sum_{i=1}^k  \int_{Q_i}  |\nabla u_i|^p \, dx
}{ \sum_{i=1}^k  \int_{Q_i} \rho(\tfrac{x}{\ve})|u_i|^p\, dx}\\
&= \beta \ve_0^{-p} \mu_{1,\frac{\ve}{\ve_0}}(Q_0).
\end{align*}
Since $v$ was arbitrary, we find that
\begin{align*}
\lambda_{k,\ve}^+ &\le \sup_{v\in C_k}\frac{\int_{\Omega} a_{\ve}(x,\nabla v)\cdot \nabla v\, dx}{\int_{\Omega} \rho(\tfrac{x}{\ve})|v|^p\, dx}\\
&\le \beta \ve_0^{-p} \mu_{1,\frac{\ve}{\ve_0}}(Q_0).
\end{align*}
Finally, since from Step 1 we have that $\mu_{1,\ve}(Q_0)\le C\ve^{-1}$, we obtain the desired result.

\begin{itemize}
\item[5.-] Case $\bar \rho < 0$. Lower bound.
\end{itemize}

Here  we work with $\sigma = \rho+c$, we add a positive constant to  $\rho$  such that
$\bar \sigma = 0$. Since
 $$
 \int_{\Omega } \rho(x) |u|^p\, dx \le \int_{\Omega} \sigma(x)|u|^p \,dx,
 $$
for any $u\in W_0^{1,p}(\Omega)$, we have the Sturmian type comparison $\lam_k^+(\rho)
\ge \lam_k^+(\sigma)$, which follows easily by comparing the Rayleigh quotients.
Together with the previous part of the proof we obtain
 $$C \ve^{-1} \le \lam_k^+(\sigma) \le \lam_k^+(\rho).$$

\begin{itemize}
\item[5.-] Case $\bar \rho < 0$. Upper bound.
\end{itemize}

Let $Q\subset [0,1]^N$ be a cube such that $\int_Q \rho^+(x)\, dx >0$. Let
$\{\mu_k^+\}$ be the positive eigenvalues of
\begin{equation}\label{aux}
\begin{cases}
-div(|\nabla u|^{p-2}\nabla u) = \mu \rho |u|^{p-2}u, & \text{in } Q\\
u=0 & \text{on } \partial Q.
\end{cases}
\end{equation}
   By using Theorem \ref{sturmian},   the variational characterization of eigenvalues together
with inequality \eqref{cont.coer.phi} and the scaling of eigenvalues, we get
 $$
 \lambda_{\ve, k}(\Omega) \le \lambda_{\ve, k}(Q_{\ve}) \le \beta \ve^{-p}\mu_{k}(Q),
 $$
and the upper bound is proved.

\begin{itemize}
\item[6.-] Case $\bar \rho >0$.
\end{itemize}

This one follows from the previous one, by changing $\rho \to -\rho$.

The proof is finished.
\end{proof}

\section{Convergence of  eigenvalues}

In this final section we prove the convergence rate of the eigenvalues in the case where the operators are independent of $\ve$ and the weights are periodic.

\begin{proof}[Proof of Theorem \ref{main2}]

Let us recall the following characterization of eigenvalues:
$$
 \frac{1}{\lambda_k} = \sup_{C\in \mathcal{C}_k} \inf_{u\in C} \int_{\Omega} \rho(x) |u|^p \, dx.
$$
where $u$ is normalized with $\|\nabla u\|_p=1$.

 Let us fix $k$ and given $\ve>0$ we can choose $C_{\ve}\in \mathcal{C}_k$ such that
$$
 \frac{1}{\lambda_{k,\ve}} \le \inf_{u\in C_{\ve}} \int_{\Omega} \rho(\tfrac{x}{\ve}) |u|^p \,dx + \ve.
$$
Now, by using that $C_{\ve}$ is compact, we can choose $\{u_i\}_{i=1}^J \subset
C_{\ve}$, and $r>0$ ($r=r(\bar \rho, p, \ve, C_\ve)$) such that
\begin{align*}
C_{\ve} \subset \bigcup_{i=i}^J B(u_i, r), &\\
\left|\int_{\Omega} \bar\rho (|u|^p - |u_i|^p) \, dx\, \right| <\ve & \qquad \mbox{ if }  u\in B(u_i, r),
\end{align*}
and we choose $u_{i_0}$ such that
 $$
\int_{\Omega} \bar\rho |u_{i_0}|^p \,dx  = \min_{1\le i \le J}
\int_{\Omega} \bar\rho |u_{i}|^p \,dx.
 $$

Therefore, we have
\begin{align*}
 \frac{1}{\lambda_{\ve, k}} &  \le   \inf_{u\in C_{\ve}} \int_{\Omega} \rho(\tfrac{x}{\ve}) |u|^p \,dx + \ve \\
&  \le    \int_{\Omega} \rho(\tfrac{x}{\ve}) |u_{i_0}|^p \,dx + \ve \\
&  =    \int_{\Omega} (\rho(\tfrac{x}{\ve}) - \bar\rho)|u_{i_0}|^p \,dx  +
  \int_{\Omega}  \bar\rho |u_{i_0}|^p \,dx + \ve \\
&  \le  \int_{\Omega}  \bar\rho |u_{i_0}|^p \,dx + O(\ve) \\
& \le \inf_{u\in C_{\ve}} \int_{\Omega}  \bar\rho |u|^p \,dx  + O(\ve) \\
& \le \sup_{C\in \mathcal{C}_k} \inf_{u\in C} \int_{\Omega}  \bar\rho |u|^p dx  + O(\ve) = \frac{1}{\lambda_{k}} + O(\ve),
\end{align*}
where the term $O(\ve)$ is given by Theorem \ref{lema.clave}.

The same arguments can be used interchanging the role of $\lambda_{\ve, k}$ and
$\lambda_{k}$.

Hence, from
$$
\left| \frac{1}{\lambda_{\ve, k}}  - \frac{1}{\lambda_{ k}}  \right| \le C \ve.
$$
Finally, using the asymptotic behavior of eigenvalues, $\lam_k \approx C k^{p/N}$,  we obtain the desired bound and the proof is finished.
\end{proof}

\section*{Acknowledgements}

This paper was partially supported by Universidad de Buenos Aires under grant UBACyT 20020130100283BA, by CONICET under grant PIP 2009 845/10 and by ANPCyT under grant PICT 2012-0153. 

\bibliographystyle{amsplain}
\bibliography{biblio}

\end{document}